\documentclass[11pt,leqno]{amsart}

\usepackage{latexsym}
\usepackage{amssymb}
\usepackage{amsmath}
\usepackage{amsthm}
\usepackage{verbatim}
\usepackage{pdfsync}
\usepackage[margin=1.2in]{geometry}
\usepackage{mathtools}

\usepackage{bm}   

\usepackage{color}
\usepackage{tikz}
\usepackage[all]{xy}

\numberwithin{equation}{section}

\newtheorem{theorem}{Theorem}
\newtheorem{lemma}[theorem]{Lemma}
\newtheorem{corollary}[theorem]{Corollary}

\newtheorem{proposition}[theorem]{Proposition}

\theoremstyle{remark}
\newtheorem{remark}{Remark}

\newcommand{\FD}{\text{\rm FD}}
\newcommand{\LDS}{\text{\rm LDS}}

\newcommand{\bpm}{\begin{pmatrix}}
\newcommand{\epm}{\end{pmatrix}}

\newcommand{\Z}{\mathbb{Z}}
\newcommand{\N}{\mathbb{N}}
\newcommand{\SL}{{\text {\rm SL}}}

\newcommand{\sgn}{\operatorname{sgn}}

\newcommand{\g}{\gamma}
\renewcommand{\H}{\mathbb{H}}
\newcommand{\R}{\mathbb{R}}
\renewcommand{\d}{\partial}
\newcommand{\C}{\mathbb{C}}
\newcommand{\e}{\varepsilon}
\newcommand{\DS}{\text{\rm DS}}
\newcommand{\Cal}[1]{\mathcal{#1}}
\renewcommand{\a}{\alpha}
\newcommand{\ord}{\text{\rm ord}}

\renewcommand{\g}{\mathfrak g}
\renewcommand{\P}{\Phi}

\newcommand{\kay}{\bm{k}}   
\newcommand{\lra}{\longrightarrow}
\newcommand{\uf}{f}
\newcommand{\tf}{\tilde{f}}
\newcommand{\ue}{e}
\newcommand{\te}{\tilde{e}}
\newcommand{\nass}{\noalign{\smallskip}}

\newcommand{\now}{\count0=\time
\divide\count0 by 60
\count1=\count0
\multiply\count1 by 60
\count2= \time
\advance\count2 by -\count1
\the\count0:\the\count2}

\begin{document}

\title{A classification of harmonic Maass forms}

\author{Kathrin Bringmann}
\address{Mathematical Institute\\University of Cologne\\ Weyertal 86-90 \\ 50931 Cologne \\Germany}
\email{kbringmann@math.uni-koeln.de}

\author{Stephen Kudla}
\address{Department of Mathematics\\University of Toronto\\40 St. George St. BA6290\\Toronto, Ontario,  M5S 2E4\\Canada}
\email{skudla@math.toronto.edu}

\begin{abstract}
We give a classification of the Harish-Chandra modules generated by the pullback to 
$\SL_2(\R)$ of harmonic Maass forms for congruence subgroups of $\SL_2(\Z)$ with exponential growth allowed at the cusps. 
We assume that the weight is integral but include vector-valued forms. Due to the weak growth condition, these 
modules do not need to be irreducible. Elementary Lie algebra considerations imply that there are 9 possibilities, and we show, by
giving explicit examples, that all of them arise from harmonic Maass forms. Finally, we briefly discuss the case of 
forms that are not harmonic but rather are annihilated by a power of the Laplacian, where much more complicated 
Harish-Chandra modules can arise. We hope that our classification will prove useful in understanding harmonic Maass forms from 
a representation theoretic perspective and that it will illustrate in the 
simplest case the phenomenon of extensions occurring in the space of automorphic forms.
\end{abstract}


\maketitle

\centerline{\bf }

{

\section{Introduction and statement of results}
The standard definition of a weight $k$ modular form $f$, say for $\SL_2(\Z)$ or a subgroup $\Gamma$ of finite index in $\SL_2(\Z)$, requires that $f$
is holomorphic at the cusps. If this condition is relaxed to allow poles at the cusps, the resulting {\it weakly holomorphic} modular forms
can have negative weight and the associated $\Gamma$-invariant function $\tf$ on $\Gamma\backslash G$, $G:=\SL_2(\R)$, does not need to be square-integrable.
By further relaxing the holomorphicity condition by allowing functions which are annihilated by the Laplace operator, one arrives at harmonic Maass forms -- see Section~\ref{section2} for the precise definition.
Such functions have
proved to be of considerable interest and importance, as they arise, for example, in the study of mock modular forms, in the constructions of Borcherds forms, in
Bruinier's construction of Green functions for divisors on orthogonal Shimura varieties, as incoherent Eisenstein series.
In a number of cases, vector-valued forms are involved, and so we include such forms in our discussion.

Suppose that $f$ is a harmonic Maass form of integral weight $k$ and let $\tf$ be the corresponding
function on $G$, cf. (\ref{lifttoG}) below. In the scalar-valued case, it is left invariant under $\Gamma$, and, in general, it satisfies the equivariance property
(\ref{gamma.inv}) below. The group $G$, and hence its
complexified Lie algebra,  acts on such functions by right translations. Let $M(\tf)$ be the $(\g,K)$-module generated by $\tf$, where $\g := \text{\rm Lie}(G)_\C$ is the complexified Lie algebra
of $G$ and $K:=\text{\rm SO}(2)$ is the standard maximal compact subgroup.

 Our first result,  Theorem~\ref{abstractMf}, is a classification of the possible $(\g,K)$-modules that could arise as $M(\tf)$'s, based on
considerations of the structure of a cyclic $(\g,K)$-module generated by a ``harmonic'' vector.
The resulting indecomposable modules are built up from
the various irreducible $(\g,K)$-modules occurring as constituents of the principal series for $G$ at points of reduciblity. It turns out that there
are $9$ possibilities. Most are subquotients of the reducible principal series but several are not, cf. Corollary~\ref{sub-quotes}.
The proof of this result is an easy exercise and the list of possibilities has a simple reformulation in terms of the behavior of $f$
under the classical raising and lowering operators, cf. Remark~\ref{remark4}.
Note that some of the possible $(\g,K)$-modules associated to harmonic Maass forms were described
in the earlier work of Schulze-Pillot (Proposition 3 of \cite{schulze-pillot}) where he assumed\footnote{In the case $k=1$,
the principal series,  $V(1)$ in his notation, is a direct sum $\LDS^+(0)\oplus \LDS^-(0)$, in our notation. So
his indecomposable in this case should be taken to mean the
indecomposable {II}(b) in our Theorem~\ref{abstractMf} rather than $V(1)$.} that $f$ has weight $k \le 1$.

Our second and main result shows that every possibility listed in Theorem~\ref{abstractMf} actually arises as an $M(\tf)$ for
some harmonic Maass form $f$.  We prove this in Section~\ref{section.examples}, by giving explicit examples for each case.
Here it should be noted that it is essential to consider vector-valued forms,  since it is shown that certain cases cannot occur non-trivially for scalar-valued forms. In fact, in realizing our list, we only use the symmetric tensor representations $(\rho_m,\mathcal P_m)$ of $G$
on the space $\mathcal P_m$ of polynomials
of degree at most $m\in\N_0$, cf. (\ref{poly-rep}) below.   For example, the $\mathcal P_m$-valued
function defined via
$$e_{r,m-r}(\tau)(X) := \frac{(-1)^{m-r}}{r!}\,v^{r-m} \, (X-\tau)^r(X-\overline{\tau})^{m-r},\qquad 0\le r\le m,$$
for $\tau=u+iv\in \H$, satisfies
$$e_{r,m-r}(\gamma\tau) = (c\tau+d)^{m-2r}\,\rho_m(\gamma)\,e_{r,m-r}(\tau),$$
for all  $\gamma=\left(\begin{smallmatrix}
a&b\\ c&d
\end{smallmatrix}\right)\in G$, and thus
the holomorphic function $e_{m,0}$ is a harmonic Maass form of weight $-m$ and type $\rho_m$ for any $\Gamma$.
The space
$$M(\widetilde{e_{m,0}}) = \text{\rm span}\left\{\widetilde{e_{m,0}}, \widetilde{e_{m-1,1}} \dots, \widetilde{e_{1,m-1}}, \widetilde{e_{0,m}}  \right\},$$
realizes the finite dimensional $(\g,K)$-module of dimension $m+1$, case {I}(a) on our list.  As already observed in \cite{schulze-pillot}, the only finite dimensional representation which
can occur for scalar-valued forms is the trivial representation.

A related case involves a generalization of the (non-holomophic) weight $2$ Eisenstein series $E_2^*$ defined in \eqref{E2star}.
This case is treated in \cite{PSS}, however, for the reader's convenience, we give all details. The $(\g,K)$-module associated to
$E_2^*$ is a non-split extension with the weight $2$ holomorphic discrete series representation as quotient and the trivial representation as submodule.
For any $m\in\N_0$, there is a $\mathcal P_m$-valued function
$$E^*_{m+2}(\tau):= \sum_{r=0}^m \frac{1}{r+1} {m\choose r}\,\ue_{r,m-r}(\tau)\,R^r\,E^*_2(\tau),$$
where $R^r=R^r_k:=R_{k+2(r-1)}\circ\ldots\circ R_{k+2}\circ R_k$ is the $r$-fold application of the Maass raising operator
\begin{equation}\label{raising}
	R=R_k:= 2i\frac{\partial}{\partial \tau} +\frac{k}{v}.
\end{equation}
It satisfies
$$
L_{m+2}\,E^*_{m+2} = \frac{3}{\pi} \,\ue_{0,m}
$$
with the Maass lowering operator
\begin{equation}\label{lowering}
	L=L_k:=-2iv^2\frac{\partial}{\partial\overline{\tau}}.
\end{equation}
Since $e_{0,m}$ is annihilated by $R_{m}$ and the Laplace operator $\Delta_k$ (defined in \eqref{Laplace} below) satisfies $\Delta_k = -R_m\circ L_{m+2}$, $E^*_{m+2}$
 is a harmonic Maass form of weight $k=m+2$. The associated
$(\g,K)$-module is a non-split extension with the weight $k$ holomorphic discrete series representation as quotient and the finite-dimensional
space $M(\widetilde{e_{m,0}})$ as submodule; this is case {III}(b) in our list.  Note that the functions $R^r\,E^*_2$ occurring in the components of $E^*_{m+2}$
lie in the spaces of nearly holomorphic modular forms in Theorem~4.2 of \cite{PSS}.

Case {II}(b) in our list is a non-split extension with the holomorphic weight $1$ limit of discrete series representation as quotient and
anti-holomorphic weight $1$ limit of discrete series representation as submodule, cf. (\ref{tiny.ext}) below. This is precisely the $(\g,K)$-module
associated to the central derivative of the incoherent Eisenstein series of weight $1$ introduced in \cite{kry.tiny}.
This weight $1$ harmonic Maass form can be viewed as the ``modular completion'' of the generating series for arithmetic degrees of special $0$-cycles
on the moduli space of CM elliptic curves, loc.\!\,cit., and hence is typical of a class of such forms whose holomorphic part is related to arithmetic.
Other such examples occur, sometimes only conjecturally, for example
in  \cite{kry.faltings,kry.book}, and in the work of Duke-Li \cite{duke.li}.  Analogous phenomena have emerged the in the theory of $p$-adic modular
forms \cite{darmon.lauder.rotger}.

\begin{remark}
In this paper, we do not include of half-integral weight forms for which a similar classification could be made, cf. \cite{schulze-pillot} for
a discussion.
\end{remark}

Finally, we note that the harmonicity condition $\Delta_kf=0$ is essential to our classification, since it implies that the
$K$-types occur in $M(\tf)$ with multiplicity $1$.  In Section~\ref{section.related}, we show that more complicated $(\g,K)$-modules
can show up if one only requires that $\Delta_k^\ell f=0$ for some $\ell\in\N_0$.  A simple, but arithmetically interesting example is given by the
weight $0$ non-holomorphic modular form
\begin{equation}\label{Kronecker-fun}
\phi(\tau) :=  -\frac16\,\log\big(|\Delta(\tau)|^2\,v^{12}\big),
\end{equation}
which arises in the Kronecker limit formula, where $\Delta$ is the weight $12$ modular discriminant.
This form is annihilated by $\Delta_0^2$ and the associated $(\g,K)$-module has the following picture:
$$
\begin{matrix}
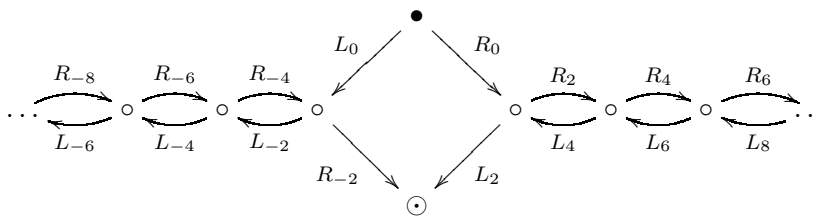
\text{\small\bf  Figure 1.  The $(\g,K)$-module for the Kronecker limit formula\qquad\quad}\\
\xymatrix{
{}&{}&{}&{}&{\bullet}\ar[dl]_{L_0}\ar[dr]^{R_0}&{}&{}&{}&{}&{}\\
\dots\ar@/^/[r]^{R_{-8}}&{\circ} \ar@/^/[l]^{L_{-6}}\ar@/^/[r]^{R_{-6}}&{\circ}\ar@/^/[l]^{L_{-4}}
\ar@/^/[r]^{R_{-4}}&{\circ}\ar[dr]_{R_{-2}}\ar@/^/[l]^{L_{-2}}&
{}&{\circ}\ar[dl]^{L_2}\ar@/^/[r]^{R_2}&{\circ}\ar@/^/[r]^{R_4}\ar@/^/[l]^{L_{4}}&\ar@/^/[l]^{L_{6}}{\circ}\ar@/^/[r]^{R_6}&\dots \ar@/^/[l]^{L_{8}}\\
{}&{}&{}&{}&{\odot}&{}&{}&{}&{}&
}
\end{matrix}
$$
An explanation is given in Section~\ref{section.related}.

In summary, we hope that our classification will prove useful in understanding harmonic Maass forms from a representation theoretic
perspective and that it will provide an elementary motivating example\footnote{In fact, this development is already underway, cf. Remark~\ref{rem10} in section ~\ref{section.related}  for a brief discussion.} for the study of extensions occurring in the space of automorphic forms, including those
with weaker than traditional growth conditions, and their analogues for more general groups. In addition, we hope that it will serve as an
accessible introduction to this point of view.

{\bf Thanks:}  This project was begun during a visit by the second author to the University of Cologne in of 2014 and completed during visits to
Oberwolfach,  TU Darmstadt, and ETH, Z\"urich, in the summer of 2016. He would like to thank these institutions for their support and stimulating working environments. The research of the first author is supported by the Alfried Krupp
Prize for Young University Teachers of
the Krupp foundation and the research leading to these results
receives funding from the European Research
Council under the European Unions Seventh Framework Programme
(FP/2007-2013) / ERC Grant agreement n.
335220 - AQSER.
The authors thank Dan Bump, Stephan Ehlen, Olav Richter, Rainer Schulze-Pillot, and Martin Westerholt-Raum for useful comments on an earlier version of this paper.

\section{Harmonic Maass form}\label{section2}

Let $(\rho,V)$ be a finite dimensional complex representation of  a subgroup $\Gamma$ of finite index in $\SL_2(\Z)$, and,
for simplicity, suppose that $k$ is an integer.
By a {\it harmonic Maass form of weight} $k$ and {\it type} $(\rho,V)$, we mean a smooth function
$\uf :\H \rightarrow  V$  satisfying the following conditions,
\cite{bruinier.funke}:
\vskip -15pt
\begin{enumerate}
	\item
	For $\gamma=\left(\begin{smallmatrix}a&b\\c&d\end{smallmatrix}\right)\in\Gamma$
	$$
	\uf (\gamma\tau) = j(\gamma,\tau)^k \,\rho(\gamma)\, \uf (\tau),
	$$
	where $j(\gamma,\tau):=c\tau+d$.
	
	\item We have
	$$
	\Delta_k \uf  = 0,
	$$
	with the {\it hyperbolic Laplacian} in weight $k$
	\begin{equation}\label{Laplace}
	\Delta_k:= -v^2\,\left(\frac{\partial^2}{\partial u^2} + \frac{\partial^2}{\partial v^2}\right) + i k\,v \left(\frac{\partial}{\partial u} + i \frac{\partial}{\partial v}\right).
	\end{equation}
	\item
	There exists a constant $B>0$ such that 
	$$\uf(\tau) = O\left(e^{B v}\right) \qquad \text{as $v\rightarrow \infty$, uniformly in $u$}.$$
	A similar condition holds at all cusps of $\Gamma$.	
\end{enumerate}
We denote this space by $H_k^{\text{mg}}(\Gamma,\rho)$.  If the representation $(\rho,V)$ is one-dimensional, i.e. is given by a character $\chi:\Gamma\rightarrow \C^\times$,
we abbreviate this to
$H_k^{\text{mg}}(\Gamma,\chi)$ or $H_k^{\text{mg}}(\Gamma)$, if $\chi$ is trivial or if we do not want to be explicit about $\chi$.
We refer to functions in this space as {\it scalar-valued Maass forms}.

A special subspace of $H_k^{\text{mg}}(\Gamma,\rho)$ consists of those forms, for which there exists a polynomial $P_f(\tau)\in V[q^{-1}]$
such that
\[
f(\tau)-P_f(\tau)=O\left(e^{-\varepsilon v}\right)
\]
as $v\to\infty$ for some $\varepsilon>0$, and similarly at the other cusps. We denote this space by $H_k(\Gamma, \rho)$
and refer to $P_f(\tau)$ as the principal part of $f$ (at the given cusp).

\begin{remark}  Note that a holomorphic function satisfying these conditions can have poles at the cusps, i.e., is a weakly holomorphic modular form in the usual,
somewhat unfortunate, terminology. Harmonic Maass forms with exponential growth at the cusps as allowed by (3) are
sometimes referred to as ``harmonic weak Maass forms''.
\end{remark}

A scalar-valued Maass form $\uf$ of weight $k \in \Z\setminus\{1\}$ has a Fourier expansion of the form (see (3.2a) and (3.2b) of \cite{bruinier.funke})
\begin{equation}\label{weak-Fourier}
f(\tau)=f^+(\tau)+f^-(\tau)
\end{equation}
where the {\it holomorphic part of $f$} is given by
\begin{equation}\label{ppart}
f^+(\tau)=\sum_{n\in\frac{1}{N}\Z \atop n\gg-\infty} c_f^+(n)\,q^n,
\end{equation}
for some $N\in \N$,
whereas, for $k\ne1$,  its {\it non-holomorphic part} is given by
\begin{equation}\label{mpart1}
f^-(\tau)=c_f^- (0)\,v^{1-k}+\sum_{\substack{n\in\frac{1}{N}\Z\setminus\{0\} \\ n\ll \infty}} c_f^-(n)\,W_k( 4\pi n v)\,q^n.
\end{equation}
Here, for $x\in \R$, $W_k(x)$ is the real-valued  incomplete gamma function\footnote{We sometimes write $\beta_k(x)$ in place of $W_k(-x/2)$, 
as this notation occurs in many places in the literature.}, defined as in \cite{BDE}, Section 2.2, by
\begin{equation}\label{def-Wk}
W_k(x) = \operatorname{Re}\big(\,\Gamma(1-k,-2x)\,\big) =\Gamma(1-k,-2x) +  \begin{cases} \frac{(-1)^{1-k} \pi i}{(k-1)!}&\text{for $x>0$,}\\
0&\text{for $x<0$,}
\end{cases}
\end{equation}
with
$$\Gamma(s,x) := \int_{x}^\infty e^{-t}\,t^{s-1}\,dt. $$

For $k=1$, the non-holomorphic part $f^-(\tau)$ has the same shape but with the term $v^{1-k}$ in (\ref{mpart1}) replaced by $-\log(v)$.
Note that the constraints on $n$ in the sums, $n\gg -\infty$ in $f^+(\tau)$ and $n\ll \infty$ in $f^-(\tau)$,  are a consequence of the growth condition (3) and the
asymptotics of $\Gamma(s,x)$.

The subspace $H_k(\Gamma)$ may be characterized as those elements of $H_k^{\operatorname{mg}}(\Gamma)$ for which $c_f^-(n)=0$ for $n\ge  0$.
For $f\in H_k(\Gamma)$, we have
\begin{equation}\label{mpart2}
f^-(\tau)=\sum_{\substack{n\in\frac{1}{N}\Z \\ n<0}} c_f^-(n)\,\Gamma\left(1-k,4\pi |n| v\right)q^{n}.
\end{equation}

We define the ``flipped space''
\begin{equation}\label{sharp}
H_k^\sharp\left(\Gamma\right):=\left\{f\in H_k^{\operatorname{mg}}\left(\Gamma\right): c_f^+(n)=0\text{ for }n<0\right\}.
\end{equation}

\begin{remark}\label{Fourier-remark} 
The Fourier expansion in the case of vector-valued forms of type $(\rho,V)$ is more subtle.
Suppose that $\Gamma=\SL_2(\Z)$, so that we have
$$f(\tau+1) = \rho\left(\begin{pmatrix} 1&1\\0&1\end{pmatrix}\right)\,f(\tau).$$
If the space $V$ admits a basis of eigenvectors of $\rho(\left(\begin{smallmatrix} 1&1\\0&1\end{smallmatrix}\right))$, then corresponding components of $f$
have a Fourier expansion as in (\ref{weak-Fourier}). This is for example the case if $(\rho,V)$ is a Weil representation as in \cite{bruinier.funke}.
On the other hand, if $(\rho,V)$ is a symmetric tensor representation, say realized on a space of polynomials as in (\ref{poly-rep}),
then there is no such basis and the Fourier series of $f$ must be defined by procedure of \cite{kuga.shimura}, which we now describe.
Suppose that the representation $(\rho,V)$ is the restriction to $\SL_2(\Z)$ of a holomorphic representation of $\SL_2(\C)$.
Letting
$$f^*(\tau) :=  \rho\left(\begin{pmatrix} 1&\tau\\0&1\end{pmatrix}\right)^{-1}\,f(\tau),$$
we have
$$
f^*(\tau+1) = \rho\left(\begin{pmatrix} 1&\tau+1\\0&1\end{pmatrix}\right)^{-1}\,\rho\left(\begin{pmatrix} 1&1\\0&1\end{pmatrix}\right)\,f(\tau)=f^*(\tau).$$
Thus $f^*$ has a Fourier expansion. The twist from $f$ to $f^*$ alters the action of the Maass operators,
\eqref{raising} and \eqref{lowering} however, and this is responsible for the
occurrence of different phenomena for vector-valued forms in certain cases, specifically cases I(a) and III(b) of Theorem~\ref{abstractMf}.
\end{remark}

Harmonic Maass forms relate in many ways to classical (weakly holomorphic) modular forms. To state the first of these connections, define the Bruinier-Funke {\it $\xi$-operator}
\begin{equation*}
	\xi_k:= 2iv^k \overline{\frac{\partial}{\partial\overline{\tau}}}.
\end{equation*}
Note that
\begin{equation*}
	\xi_k f=v^{k-2}\overline{L_{k}f}=R_{-k}\left(v^k \overline{f}\right)
\end{equation*}
from which one can easily conclude that $\xi_k: H^{\operatorname{mg}}_k \left(\Gamma,\rho\right) \to M_{2-k}^!(\Gamma,\bar\rho)$, where
$\bar\rho$ is the representation of $\Gamma$ on $V$ defined by
$$\bar\rho(\gamma)v = \overline{\rho(\gamma)\bar v}.$$
Here we need to assume that $V$ is defined over $\R$, i.e., that complex conjugation $\bar{}:V\rightarrow V$ is defined.  Of course this is true for spaces of
complex-valued functions, e.g., polynomials or group algebras as in \cite{bruinier.funke}.
We also note that, if $\uf$ has weight $k$, then $v^k\,\overline{f(\tau)}$ has weight $-k$ and
\begin{equation*}
L_{-k}^a\left(v^k\,\overline{f(\tau)}\right) = v^{k+2a} \,\overline{R_k^af(\tau)}, \qquad R_{-k}^a\left(v^k\,\overline{f(\tau)}\right) = v^{k-2a}\overline{L_k^a f(\tau)} .
\end{equation*}

Now assume that $f$ is scalar-valued.  Writing the Fourier expansion of $f$ as in \eqref{weak-Fourier}, we have 
\begin{equation}\label{xiact}
	\xi_kf(\tau)
	=(1-k)\overline{c_f^-(0)}-(4\pi)^{1-k} \sum_{\substack{n\in\frac1N\Z\setminus\{0\}\\n\gg -\infty}}\overline{\frac{c_f^-(-n)}{n^{k-1}}}q^n.
\end{equation}
while, for $k=1$, we have
\begin{equation}\label{xiact}
	\xi_1f(\tau)
	=-\overline{c_f^-(0)}-\sum_{\substack{n\in\frac1N\Z\setminus\{0\}\\n\gg -\infty}}\overline{c_f^-(-n)}\,q^n.
\end{equation}
Using this operator, $H_k(\Gamma)$ may be characterized as those elements in $H_k^{\text{mg}}(\Gamma)$ which map to cusp forms under $\xi_k$.

A further operator which relates harmonic Maass forms to (weakly) holomorphic modular forms is given by iterated differentiation.  Suppose that 
$k\in-\N_0$ and let
$$
D^{1-k}:=\left(\frac{1}{2\pi i} \frac{\partial}{\partial \tau}\right)^{1-k}.
$$
Using Bol's identity, \cite{bol, bump.choie},
\begin{equation}\label{classic-bol}
D^{1-k}=(-4\pi)^{k-1}R_k^{1-k},
\end{equation}
one can show that $D^{1-k}:H_k^{\operatorname{mg}}\left(\Gamma,\rho\right)\to M_{2-k}^{!}\left(\Gamma,\rho\right)$.
For $f$ scalar-valued, the operator $D^{1-k}$ acts on \eqref{weak-Fourier} as
\begin{equation}\label{D-action}
D^{1-k}f(\tau)=(1-k)!(4\pi)^{k-1}c_f^-(0)+\sum_{\substack{n\in\frac1N\Z\setminus\{0\}\\n\gg -\infty}}\frac{c_f^+(n)}{n^{k-1}}q^n.
\end{equation}
Using this operator, the space \eqref{sharp} may be characterized as
$$
H^\sharp_k\left(\Gamma\right)=\left\{f\in H_k^{\operatorname{mg}}\left(\Gamma\right):D^{1-k}(f)\in S_{2-k}\left(\Gamma\right)\right\}.
$$

Finally, we also require an operator, which ``flips'' the two spaces $H_k(\Gamma)$ and $H_k^\sharp(\Gamma)$. To be more precise, define for $f\in H_k^{\operatorname{mg}}(\Gamma)$ the {\it flip of $f$}
\begin{equation*}
	\mathfrak{F}_k:=\frac{v^{-k}}{(-k)!}\overline{R_k^{-k}}.
\end{equation*}
The flipping operator $\mathfrak{F}_k$ satisfies
\begin{align}\label{involution}
	\mathfrak{F}_k:&H_k\left(\Gamma\right)\to H_k^\sharp\left(\Gamma\right),~ H_k^\sharp\left(\Gamma\right)\to H_k\left(\Gamma\right),\notag\\
	&\qquad\mathfrak{F}_k\circ\mathfrak{F}_k(f)=f.
\end{align}
Moreover we have
\begin{align}
	\xi_k \circ\mathfrak{F}_k
	&=-\frac{(-4\pi)^{1-k}}{(-k)!}D^{1-k}, \label{xiflip}\\
	D^{1-k}\circ \mathfrak{F}_k
	&=\frac{(-k)!}{(4\pi)^{1-k}}\xi_k.\label{Dflip}
\end{align}

Natural harmonic Maass forms can be given via Poincar\'e series. For simplicity, we restrict to scalar-valued forms. We now describe the general construction. Let
\begin{equation*}
	\mathbb{P}_k(\varphi;\tau) := \sum_{\gamma =\left(\begin{smallmatrix} a & b \\ c & d \end{smallmatrix}\right)\in\Gamma_\infty\setminus\SL_2(\Z)}\varphi \Big|_k\gamma(\tau),
\end{equation*}
where $\Gamma_\infty:=\{\pm\left(\begin{smallmatrix}1&n\\0&1\end{smallmatrix}\right):n\in\Z\}$, $\big|_k$ is the usual weight $k$ slash operator and $\varphi : \H \rightarrow \C$ satisfies
\begin{equation*}
\varphi(\tau) = O(v^{2-k+\e})\qquad (\e>0).
\end{equation*}
For $k<0$, let
\begin{equation*}
	F_{k,m}(\tau) := \mathbb{P}_k\left(\varphi_{k,m}\right),
\end{equation*}
with $(e(u):=e^{2\pi iu})$
\begin{equation*}
	\varphi_{k,m}(\tau) := \frac{(-\sgn(m))^{1-k}}{(1-k)!}\left(4\pi |m|v\right)^{-\frac{k}{2}}M_{\sgn(m)\frac{k}{2},\frac{1-k}{2}}\left(4\pi|m|v\right)e(mu)
\end{equation*}
with $M_{\mu,\nu}$ the $M$-Whittaker function. These functions give rise to harmonic Maass forms (see e.g. \cite{Br, Ni}).
To be more precise, for $m>0$, we have
\[
	F_{k,m} \in H_k,\quad
	F_{k,-m} \in H^{\sharp}_k.
\]

\section{The $(\g,K)$-module defined by $f$.}
As usual, for $g\in G = \SL_2(\R)$ and $f\in H_k^{\text{mg}}(\Gamma,\rho)$, we let
\begin{equation}\label{lifttoG}
\tf(g) := j(g,i)^{-k}\, \uf (g(i)),
\end{equation}
so that $\tf:G \rightarrow V$ satisfies 
\begin{align}\label{gamma.inv}
	\tf(\gamma g) &= \rho(\gamma) \tf(g) \qquad\text{for all $\gamma\in\Gamma$},\\
\label{K-type}
\tf(g k_\theta) &= e^{i k \theta}\,\tf(g)\qquad\text{for all $k_\theta :=
\left(\begin{smallmatrix} \cos(\theta)&\sin(\theta)\\-\sin(\theta)&\cos(\theta)\end{smallmatrix}\right) \in K=\text{\rm SO}(2)$,}
\end{align}
and the growth condition
\begin{equation}\label{growth.cond}
\tf(g) = O\left(e^{Bv}\right), \qquad g(i) = u+iv,
\end{equation}
for a constant $B>0$, uniformly in $u$, as $v\rightarrow \infty$.
Again, similar conditions hold at all cusps of $\Gamma$.
As a good reference for this construction as well as the calculations to follow is \cite{verdier}, see also Chapter~1 of \cite{vogan.book}.

Let
$$
H := i \bpm 0&-1\\1&0\epm, \qquad X_+ := \frac12 \bpm 1&i \\ i&-1\epm, \qquad  X_- := \frac12 \bpm 1&-i \\ -i&-1\epm,
$$
be the standard basis for the complexified Lie algebra $\g$ of $G$.  Note that $H$ spans $\mathfrak k$, the complexified Lie algebra of $K$, and that condition (\ref{K-type}) on $\tf$ is equivalent to the condition $H \tf = k \tf$.
Recall that, if $\tf$ is the lift of $\uf $ to $G$, then $X_+ \tf$ is the lift of $R_k \uf $ and $X_-\tf$ is the lift of $L_k\uf $, where $L_k$ and $R_k$ are defined in \eqref{raising} and \eqref{lowering}, respectively.

The Casimir operator is the element of the universal enveloping algebra $U(\g)$ of $\g$ defined by
\begin{equation}\label{casimir}
C:= H^2 + 2 X_+ X_- + 2 X_- X_+
= (H-1)^2 + 4 X_+X_--1
= (H+1)^2 + 4 X_-X_+-1.
\end{equation}
Note that $C$ spans the center of $U(\g)$ and acts by a scalar in any irreducible $(\g,K)$-module. 
Using the second expression for $C$,  it is immediate that the condition $\Delta_k\uf=0$ is equivalent to
\begin{equation}\label{Casimir.cond}
C\tf = \left(\left(k-1\right)^2-1\right)\tf.
\end{equation}
Note that, for $k\in -\N_0$, the flipping operator lifts to
$$\widetilde{\mathfrak{F}_kf} =\frac{1}{(-k)!}\,\overline{X_+^{-k}\tilde{f}},$$
 and the identity (\ref{involution}) amounts to
 $$X_-^{-k}\,X_+^{-k}\,\tilde{f} = (-k)!^2\,\tilde f.$$

Let $A(G,V)$ be the space of all $K$-finite\footnote{A function in this space is a finite linear combination of functions satisfying
(\ref{K-type}) for various weights $k$.}, $C^\infty$-functions on $G$, which are valued in $V$, and let
$A(G,V;\Gamma)$ be the subspace such that (\ref{gamma.inv}) holds. Note that this subspace is a $(\g,K)$ submodule.
For a harmonic Maass form $\uf$ of weight $k$, the corresponding function $\tf$ on $G$ lies in $A(G,V;\Gamma)$ and satisfies
conditions (\ref{K-type}) and (\ref{Casimir.cond}), as well as the growth condition (\ref{growth.cond}), which we do not use for the moment.
We want to describe the $(\g, K)$-submodule $M(\tf)$  of $A(G,V)$ generated by $\tf$.

For $j= k\pm2r\in k+2\Z$, $r\in \N_0$, let
\begin{equation}\label{first-up-down}
\tf_j := X_\pm^r\,\tf.
\end{equation}
The following fact is classically well-known. 

\begin{proposition}
	The $(\g,K)$-submodule $M(\tf)$  of $A(G,V)$ generated by $\tf= \tf_k$ is spanned by the $\tf_j$.  Moreover, for $r\in\N$,
	\begin{align}
		X_- \tf_{k+2r} &=  r(1-k-r)\,\tf_{k+2(r-1)},\quad{\rm and}\label{movedown}\\
		X_+ \tf_{k-2r} &= -(r-1)(r-k)\tf_{k-2(r-1)}.\notag
	\end{align}
	In particular,  the following vanishing always occurs:
	\begin{align}
		X_+\tf_{k-2} &=0, \notag\\
		X_- \tf_{2-k} &=0  &&\text{if $k<1$ {\rm \text{\quad (so that $r = 1-k >0$ in (\ref{first-up-down})})}},\label{autovan2}\\
		X_+\tf_{-k}&=0 &&\text{if $k>0$ {\rm \text{\quad (so that $r = k>0$ in (\ref{first-up-down})})}}.\label{autovan3}
	\end{align}
\end{proposition}
\begin{proof}  For convenience, we set $\nu= 1-k$.
	By (\ref{casimir}), we have
	\begin{align*}
		4 X_+X_- & = C - (H-1)^2 +1,\\
		4 X_-X_+ &= C- (H+1)^2+1,
	\end{align*}
	and hence, recalling that $C$ is in the center of the enveloping algebra,
	\begin{align}
		4 X_+X_- \,\tf_j &= \left(\left(\nu^2-1\right) -(j-1)^2+1\right)\,\tf_j,  \label{multi-one-1}\\
		4 X_-X_+ \,\tf_{j-2}&= \left(\left(\nu^2-1\right) -(j-1)^2+1\right)\,\tf_{j-2}.  \label{multi-one-2}
	\end{align}
	Now for $j= k+2r>k$, the second equation gives us
\[
	X_- \tf_{j} =  \frac14\left(\left(\nu^2-1\right) -(2r-\nu)^2+1\right) \tf_{j-2},\quad\text{i.\,e.,}\quad
	X_- \tf_{k+2r} =  r(\nu-r)\, \tf_{k+2(r-1)}.
	\]
	Similarly, for $r>0$, the first equation yields
	$$
	X_+ \tf_{k-2r} = -(r-1)(\nu+r-1) \tf_{k-2(r-1)}.
	$$
\end{proof}

\begin{remark}\label{K-type-remark}
 Note that it is precisely the harmonicity condition (\ref{Casimir.cond}) which yields the relations (\ref{multi-one-1})
and (\ref{multi-one-2}).  These imply that the $K$-types in $M(\tf)$ have multiplicity $1$, a fact which is
crucial to our determination of this module. If condition (\ref{Casimir.cond}) is relaxed, higher multiplicities and more complicated things can occur
as shown by some examples in Section~\ref{section.related}.
\end{remark}
\begin{remark} It is useful to note that for $k<1$ and $\nu=1-k$, the vanishing $X_- \tf_{2-k} =0$  in (\ref{autovan2})
implies that, for a holomorphic form $f$,
\begin{equation}\label{trick}
R^\nu_k \uf (\tau) =\left(2 i \frac{\partial}{\partial \tau}\right)^{\nu}\uf(\tau),
\end{equation}
i.e., Bol's identity (\ref{classic-bol}).
\end{remark}

\section{Some standard $(\g,K)$ modules}
Before continuing our analysis, we review the structure of the reducible principal series
representations of $G$. This provides the irreducible $(\g,K)$-modules from which $M(f)$ is built. A nice reference for this material is Chapter~1 of
\cite{vogan.book}. Here we use the notation
$$m(a) := \begin{pmatrix} a&{0}\\{0}&a^{-1}\end{pmatrix}, \qquad n(b) := \begin{pmatrix} 1&b\\{0}&1\end{pmatrix}, \quad\text{and}\quad
k_\theta = \begin{pmatrix} \cos(\theta)&\sin(\theta)\\-\sin(\theta)&\cos(\theta)\end{pmatrix}.$$
For $\varepsilon\in \{0,1\}$ and $\nu\in \C$, let $I^{\text{\rm sm}}(\varepsilon,\nu)$,  be the principal series representation of $G$
given by right multiplication on the space of smooth functions\footnote{Hence the superscript `sm'.}
$\phi$ such that
$$
\phi(n(b) m(a) g) = \sgn(a)^\varepsilon\,|a|^{\nu+1}\,\phi(g).
$$
Let $I(\varepsilon,\nu)$ be the corresponding $(\g,K)$-module of $K$-finite functions.
For $j \equiv \varepsilon\pmod{2}$, let
$\phi_j\in I(\varepsilon,\nu)$ be the function such that
\begin{equation}\label{ind-K-type}
\phi_j(k_\theta) = e^{j \theta i}.
\end{equation}
These functions give a basis for $I(\e,\nu)$, and and easy calculation shows that
\begin{equation}\label{PS.action}
	X_{\pm} \phi_j = \frac12(\nu+1\pm j)\phi_{j\pm2}.
\end{equation}
From this it is immediate that 
\begin{equation}
C \phi_j =\Big((j-1)^2 + (\nu+1-j) (\nu+1+j-2)-1\Big)\phi_j = \left(\nu^2-1\right) \phi_j.
\end{equation}
The $(\g,K)$-module $I(\varepsilon,\nu)$ is irreducible unless $\nu$ is an integer and $\nu-1\equiv \varepsilon\pmod{2}$. In that case,  $I(\varepsilon,\nu)$ is not irreducible and $\nu$ determines $\varepsilon$, so we write simply ${I(\nu) = I(\varepsilon,\nu)}$. The structure of $I(\nu)$ is well-known and can easily be derived from (\ref{PS.action}).

\begin{enumerate}
	\item
	If $\nu>0$, then there are two irreducible submodules, the holomorphic (resp. anti-holomorphic) discrete series representation
	$$
	\DS^{+}(\nu) = [\phi_{\nu+1}, \phi_{\nu+3}, \dots, ],\qquad(\text{resp. }
	\DS^-(\nu) = [\dots, \phi_{-\nu-3},\phi_{-\nu-1}])
	$$
	with lowest (resp. highest) weight $k$ (resp. $-k$), and a unique irreducible quotient
	\begin{equation}\label{Is0-decomp}
	\FD(\nu) = [\phi_{-\nu+1}, \dots, \phi_{\nu-1}] \pmod{ \DS^+(\nu)\oplus \DS^-(\nu)}
	\end{equation}
	so that there is an exact sequence
	$$0\lra \DS^+(\nu)\oplus \DS^-(\nu) \lra I(\nu) \lra \FD(\nu)\lra 0.$$
	In particular, $\FD(\nu)$ is the finite dimensional representation of $G$ of dimension $\nu$.
	
	\item
	If $\nu<0$, then there is a unique irreducible submodule
	$$
	\FD(-\nu) = [\phi_{\nu+1}, \dots, \phi_{-\nu-1}],
	$$
	and two irreducible quotients, $\DS^{\pm}(-\nu)$, and so an exact sequence
	$$0\lra \FD(-\nu)\lra I(\mu) \lra \DS^+(-\nu)\oplus \DS^-(-\nu)\lra 0.$$
	
	\item
	If $\nu=0$, there are two irreducible summands, the limits of discrete series
	$$
	\LDS^+(0) = [\phi_1,\phi_3,\dots]\qquad\text{and}\qquad
	\LDS^-(0) = [\dots, \phi_{-3}, \phi_{-1}],
	$$
	with lowest (resp. highest) weight $1$ (resp. $-1$).
\end{enumerate}

In each of the cases, we have, with $k=1-\nu$,
$$
X_+^r \phi_k = r! \,\phi_{k+2r},\quad\text{and}\quad
X_-^r\phi_k = \nu(\nu+1)(\nu+2) \cdots (\nu+r-1)\,\phi_{k-2r},
$$
and, an easy calculation gives
$$
X_- X_+^r \phi_k = r(\nu-r) \, X_+^{r-1}\,\phi_k, \quad\text{and}\quad
X_+X_-^r\phi_k =  -(r-1)(\nu+r-1)\,X_-^{r-1}\phi_k.
$$

\section{A Classification}
Now returning to the subspace $M(\tf)$ of $A(G,V)$ generated by $\tf=\tf_k$, we distinguish the three cases $k<1$, $k=1$, and $k>1$.
The transition equations (\ref{movedown})--(\ref{autovan3}) then show when it is possible to move up and
down among the $\tf_j$'s and hence reveal the possible module structures for $M(\tf)$ as an abstract $(\g,K)$-module.
It turns out that there are 9 cases. Recall that $\nu=1-k$.

\begin{theorem}\label{abstractMf}\hfill\break
                {\bf I.}  Suppose that $k<1$, so that $\nu> 0$. Then the structure of $M(\tf)$ is
                determined by the functions $\tf_{k-2}$ and $\tf_{k+2\nu} = \tf_{2-k}$.  More precisely, $M(\tf)$ is isomorphic to
		\begin{enumerate}
			\item[(a)]  $\FD(\nu)$  if $\tf_{k-2}=0$ and $\tf_{k+2\nu}=0$,
			\item[(b)] $I(\nu)/\DS^-(\nu)$  if $\tf_{k-2}= 0$ and $\tf_{k+2\nu}\ne0$, so that there is an exact sequence
			$$0\lra \DS^+(\nu)\lra M\left(\tf\right) \lra \FD(\nu)\lra 0,$$
			\item[(c)]  $I(\nu)/\DS^+(\nu)$ if $\tf_{k-2}\ne 0$ and $\tf_{k+2\nu}=0$, so that there is an exact sequence
			$$0\lra \DS^-(\nu)\lra M\left(\tf\right) \lra \FD(\nu)\lra 0,$$
			\item[(d)] $I(\nu)$ if $\tf_{k-2}\ne 0$ and $\tf_{k+2\nu}\ne 0$, so that there is an exact sequence
			$$0\lra \DS^+(\nu)\oplus \DS^-(\nu)\lra M\left(\tf\right) \lra \FD(\nu)\lra 0.$$
			\item[]  Moreover, the sequences in cases {\rm(b)}, {\rm(c)}, and {\rm(d)} are not split.
		\end{enumerate}
		{\bf II.} Suppose that $k=1$ so that $\nu=0$.
		\begin{enumerate}
			\item[(a)]
			If $\tf_{-1}=0$, then $M(\tf) \simeq \LDS^+(0)$.
			
			\item[(b)]
			If $\tf_{-1}\ne 0$, then there is a non-split extension
			$$
			0\lra \LDS^-(0) \lra  M\left(\tf\right) \lra  \LDS^+(0)\lra  0.
			$$
		\end{enumerate}

		{\bf III.} Suppose that $k>1$ so that $\nu<0$. Then the structure of $M(\tf)$ is determined by the functions $\tf_{k-2}$ and $\tf_{-k}$.
		\begin{enumerate}
			\item[(a)]
			If $\tf_{k-2}=0$, then $M(\tf)\simeq \DS^+(-\nu)$.
			
			\item[(b)]
			If $\tf_{k-2}\ne 0$ and $\tf_{-k}=0$, then there is a non-split extension
			$$
			0\lra  \FD(-\nu)\lra  M\left(\tf\right) \lra  \DS^+(-\nu)\lra  0
			$$
			so that $M(\tf)$ is isomorphic to the submodule of $I(\nu)$ generated by $\phi_k$.
			
			\item[(c)]
			If $\tf_{k-2}\ne 0$ and $\tf_{-k}\ne 0$, then there is a (socle) composition series of length $3$,
			$$
			F^2M\left(\tf\right) \subset F^1M\left(\tf\right) \subset M\left(\tf\right),
			$$
			with $F^2M(\tf) \simeq \DS^-(-\nu)$, $F^1M(\tf)/F^2M(\tf) \simeq \FD(-\nu)$, and $M(\tf) /F^1M(\tf) \simeq \DS^+(-\nu)$.
		\end{enumerate}
\end{theorem}
Theorem \ref{abstractMf} immediately implies
\begin{corollary}\label{sub-quotes}
\begin{enumerate}
		\item[(i)]
		If $k<1$, i.e., in cases I(a)-(d),    $M(\tf)$ is isomorphic to a quotient of $I(\nu)$.
		
		\item[(ii)]
		In cases III(a) and III(b), where $k>1$, and in case II(a), where $k=1$, $M(\tf)$ is isomorphic to a subquotient of $I(\nu)$.
		
		\item[(iii)]
		In cases III(c), where $k>1$,  and II(b), where $k=1$, $M(\tf)$ is not isomorphic to a subquotient of $I(\nu)$.
	\end{enumerate}
\end{corollary}

\begin{remark}\label{remark4}
For convenience, we summarize what the various cases amount to in classical language.  Suppose that $\uf $ is a harmonic Maass form of weight $k$, and, for $r\in\mathbb{N}_0$, set $\uf _{k+2r} := R_{k}^r\uf $ (resp.  $\uf _{k-2r}:= L_k^r\uf $)
for the image of $\uf $ under the $r$-fold application of the raising (resp. lowering) operator.
\begin{enumerate}
\item[\bf I.]  Here $k<1$. The subcases correspond to the following:
\begin{enumerate}
\item[(a)]  $L_k\uf=\uf _{k-2}=0$ and $R_k^{1-k}\uf=\uf_{2-k}=0$, 
\item[(b)]  $L_k\uf=\uf_{k-2}= 0$ and $R_k^{1-k}\uf=\uf_{2-k}\ne0$,
\item[(c)]  $L_k\uf=\uf_{k-2}\ne 0$ and $R_k^{1-k}\uf = \uf_{2-k}=0$,
\item[(d)]  $L_k\uf = \uf_{k-2}\ne 0$ and $R_k^{1-k}\uf=\uf_{k+2\nu}\ne 0$.
\end{enumerate}
\item[\bf II.]  Here $k=1$. The subcases correspond to the following:
\begin{enumerate}
	\item[(a)]
	$L_1\uf=0$, i.\,e., $\uf$ is a weakly holomorphic modular form of weight $1$,

	\item[(b)]
	$L_1\uf \ne 0$,  i.\,e., $\xi_1\uf$ is a weakly holomorphic modular form of weight $1$.
\end{enumerate}

\item[\bf III.]  Here $k>1$. The subcases correspond to the following:
\begin{enumerate}
\item[(a)]  $L_k\uf=\uf_{k-2}=0$, i.\,e., $\uf$ is a
weakly holomorphic modular form of weight $k$,
\item[(b)]  $L_k\uf=\uf_{k-2}\ne 0$ and $L_k^{k}\uf=\uf_{-k}=0$. i.\,e., $L_k^{k-1}\uf$ is holomorphic of weight $2-k$,
\item[(c)]  $L_k\uf=\uf_{k-2}\ne 0$ and $L_k^{k}\uf= \uf_{-k}\ne 0$.
\end{enumerate}

\end{enumerate}

\end{remark}

Of course, the cases listed in the theorem are simply possibilities. Our goal is to determine which
of them can actually arise as subspaces of $A(G,V;\Gamma)$ and for which $(\rho,V)$. Of course, the cases
I(a) and II(a), where $k\in\N$, arise as irreducible submodules $M(\tf)$ generated by the lifts to $G$ of holomorphic cusp forms of weight $k$.
On the other hand,
all other cases involve non-split extensions and hence cannot occur as subspaces of $L^2(\Gamma\backslash G)$.
In fact, we show the following:

\begin{theorem}  All possible cases enumerated in Theorem~\ref{abstractMf} arise as $M(\tf)$'s.
\end{theorem}

\begin{remark}
This result is proved in Section~\ref{section.examples} by explicit construction. It is shown there, that certain cases can only occur for $(\rho,V)$
with $\dim (V)>1$.
\end{remark}

\section{Examples}\label{section.examples}

In this section, we provide examples for each of the possibilities for $M(\tf)$  listed in Theorem~\ref{abstractMf}.

\subsection{Case I: $k<1$}
\subsubsection{Case I(a)}
In this case, we want an automorphic realization of the finite-dimensional representation $\FD(\nu)$ of dimension $\nu$. For $k=0$ and $\nu=1$ the constant function gives a trivial example for the one-dimensional space $M(\tf)=\FD(1)$.
Moreover, as remarked by Schulze-Pillot \cite{schulze-pillot}, 
this is the only possibility of a finite-dimensional $M(\tf)$ if the representation $(\rho,V)$ is a character, i.e., for scalar-valued modular forms. For convenience of the reader, we provide a proof of this statement, which follows the approach of Schulze-Pillot.
\begin{lemma}\label{lemma6.5}
	Suppose that $k\le 0$ and that $\uf\in H_k^{\rm{mg}}(\Gamma)$ is a scalar-valued harmonic Maass form 
	with $M(\tf) = \FD(\nu)$.
	Then $\uf $ is a constant and $\nu=1$.
\end{lemma}
\begin{proof} The condition $L_kf=0$ implies that $f$ is holomorphic.
Then the condition $D^\nu f=0$ implies that $f$ is a polynomial. Since $f$ is invariant under $\tau\mapsto \tau+1/N$, $f$ must be constant.
\end{proof}

For vector-valued forms, each $\FD(\nu)$ can occur, as shown by the
following elementary construction, cf. \cite{verdier}.
For $k\in -\N_0$, let $m= -k$ and  let $\Cal P_m$ be the space of polynomials of degree at most $m$ in the variable $X$. The group $\SL_2(\R)$ acts of $\Cal P_m$
via $\left(\gamma:=
\left(\begin{smallmatrix}
a & b \\ c & d
\end{smallmatrix}\right)
\right)$

\begin{equation}\label{poly-rep}
\rho_m(\gamma)p(X) = (-c X+a)^{m} p\left(\frac{d X-b}{-c X+a}\right).
\end{equation}
We abbreviate $\rho:=\rho_m$. Following
\cite{verdier}, for an integer $r$ with $0\le r\le m$, define the function $\ue_{r,m-r}:\H \rightarrow  \Cal P_m$ by
\begin{align*}
\ue_{r,m-r}(\tau)(X) :&= \frac{(-1)^{m-r}}{r!}\,v^{r-m}\,\det\bpm X&\tau\\ 1 &1\epm^{r}\, \det\bpm X&\overline{\tau}\\ 1&1\epm^{m-r}\\
&= \frac{(-1)^{m-r}}{r!}\,v^{r-m}\,(X-\tau)^r \,(X-\overline{\tau})^{m-r}.
\end{align*}
Then, for any $\gamma\in \SL_2(\R)$,
\begin{equation}\label{any-Gamma}
\ue_{r,m-r}(\gamma\tau) = (c\tau+d)^{m-2r}\,\rho(\gamma)\,\ue_{r,m-r}(\tau),
\end{equation}
so that $\ue_{r,m-r}$ has weight $m-2r$.
The holomorphic function $\ue_{m,0}$ is a harmonic Maass form of weight $-m$
and type $\rho_m$. The corresponding functions $\te_{r,m-r}$ on $G$ are given by
\begin{align*}
	\te_{r,m-r}(g)(X) &= \frac{(-1)^{m-r}}{r!}\,j(g,i)^r \, j(g,-i)^{m-r}\, \det \bpm X&g(i)\\
	1&1\epm^{r}\,\det \bpm X&g(-i)\\
	1&1\epm^{m-r}.
\end{align*}

Let
$$
\phi(g) := \det\left( \bpm X\\1\epm\!, g\bpm \pm i\\1\epm\right)^r.
$$
If $A \in \g_0$, the real Lie algebra of $G$, then
$$
A \phi(g) = r\,\det\left( \bpm X\\
1\epm, g\bpm \pm i\\1\epm\right)^{r-1}\,\det\left( \bpm X\\
1\epm, A\bpm \pm i\\1\epm\right),
$$
and the same formula holds for $A\in \g$, the complexification of $\g_0$,  by linearity. In particular, since
$$
X^{+} \bpm i\\1\epm = -\bpm -i\\1\epm, \quad{\text{and}}\quad X^{+} \bpm -i\\1\epm = 0,
$$
we see that
\begin{align*}
X_+ \te_{r,m-r} = \te_{r-1,m-r+1},\qquad X_+ \te_{0,m}=  0, \quad \text{ and} \\
X_- \te_{r,m-r} =  (r+1)\,(m-r)\,\te_{r+1,m-r-1}, \qquad X_- \te_{m,0} =  0.
\end{align*}
The classical functions $\ue_{r,m-r}$ behave in the same way under raising and lowering as the $\te_{r,m-r}$ behave under $X^{\pm}$, viz
\begin{equation*}
L_{m-2r} \,\ue_{r,m-r} = (r+1)\,(m-r)\,\ue_{r+1,m-r-1},\qquad R_{m-2r} \,\ue_{r,m-r}= \ue_{r-1,m-r+1}.
\end{equation*}
In particular,
\begin{equation*}
L_{-m}\ue_{m,0}=0,\qquad R_{m}\ue_{0,m} =0.
\end{equation*}
These formulas are easily checked by a classical calculation as well.
In this way, we obtain a realization of $\FD(\nu)$ in the space  $A(G,\Cal P_m;\Gamma)$, where $m=\nu-1$ and for
any $\Gamma \subset \SL_2(\R)$; the transformation law under $\Gamma$ follows from (\ref{any-Gamma}).

Finally, we note that, under the flipping operator,
$$\mathfrak{F}_{-m}e_{m,0} = (-1)^m\,e_{m,0}.$$

\subsubsection{Case I(b)}
Any weakly holomorphic modular form $f$ of weight $k<1$ gives an extension of the form
$$0\lra \DS^+(\nu)\lra M\left(\tf\right) \lra \FD(\nu)\lra 0,$$
where $\nu=1-k$, as usual.  The submodule is generated by $R_k^{\nu}f$, which is again holomorphic by
either (\ref{autovan2})
or (\ref{trick}).

For example, for the form $f = 1/\Delta$ of weight $-12$, where $\Delta$ is the usual cusp form of weight $12$ for $\SL_2(\Z)$, 
we have $\nu=13$ and, recalling (\ref{trick}),
$$h=R^{13}_{-12} f = (2i)^{13}\,f^{(13)}$$
is a weakly holomorphic form of weight $14$.  Then the extension is 
$$0\lra M\left(\tilde h\right) \lra M\left(\tf\right) \lra \FD(13)\lra 0.$$

As an even simpler example,  we can take $f = j$, the modular $j$-invariant, with $k=0$ and $\nu=1$. Then we have $h:= R_0f= 2 i \,j'$ and we get an extension
$$0\lra M\left(\tilde h\right) \lra M\left(\tf\right) \lra \FD(1)\lra 0$$
with the trivial representation as quotient.

\subsubsection{Case I(c)}

Let $F\in M_k^{!}\backslash\{0\}$ and set $G:=\mathfrak{F}_kF$. If  $k=0$, we assume that $F$ is not a constant. 
We compute, using \eqref{xiflip},
\begin{align*}
	\xi_kG
	=-\frac{(-4\pi)^{1-k}}{(-k)!}D^{1-k}F\neq 0
\end{align*}
since $F\neq0$ (resp. $F$ is non-constant for $k=0$). Moreover, by \eqref{Dflip},
\begin{align*}
	D^{1-k}G
	=\frac{(-k)!}{(4\pi)^{1-k}}\xi_k F
	=0
\end{align*}
since $F\in M_k^!$. This gives the claim.

These last two cases, say for $\frac{1}{\Delta}$ and its flip $\mathfrak{F}_{-12}\frac{1}{\Delta}$, can be pictured as follows:  For case I(b)
\begin{equation}\label{5.6}
\xymatrix{
{}&{\bullet} \ar@/^/[r]^{R_{-12}}&{\circ}\ar@/^/[l]^{L_{-10}}
\ar@/^/[r]^{R_{-10}}&\ \dots\ \ar@/^/[l]^{L_{-8}}\ar@/^/[r]^{R_8}&
{\circ}\ar@/^/[r]^{R_{10}}\ar@/^/[l]^{L_{10}}&\ar@/^/[l]^{L_{12}}{\circledast}\ar@/^/[r]^{R_{12}}&
\odot\ar@/^/[r]^{R_{14}}&{\circ}\ar@/^/[r]^{R_{16}}\ar@/^/[l]^{L_{16}}&{\circ}\ar@/^/[l]^{L_{18}}\ar@/^/[r]^{R_{18}}&\cdots\ar@/^/[l]^{L_{20}}
},
\end{equation}
where $\bullet$ indicates $\frac{1}{\Delta}$, $\circledast$ indicates the form $R_{-12}^{12}\frac{1}{\Delta}$ of weight $12$,  and $\odot$ indicates the weight $14$ holomorphic form $h$. The omitted arrows are zero.

Applying the flip yields the example for case I(c): 
\begin{equation}\label{5.7}
\xymatrix{
\dots\ar@/^/[r]^{R_{-20}}&{\circ}\ar@/^/[r]^{R_{-18}}\ar@/^/[l]^{L_{-18}}&{\circ}\ar@/^/[r]^{R_{-16}}\ar@/^/[l]^{L_{-16}}&{\breve{\odot}}\ar@/^/[l]^{L_{-14}}&
{\breve\circledast} \ar@/^/[r]^{R_{-12}}\ar@/^/[l]^{L_{-12}} &{\circ}\ar@/^/[l]^{L_{-10}}
\ar@/^/[r]^{R_{-10}}&\ \dots\ \ar@/^/[l]^{L_{-8}}\ar@/^/[r]^{R_8}&
{\circ}\ar@/^/[r]^{R_{10}}\ar@/^/[l]^{L_{10}}&\ar@/^/[l]^{L_{12}}{\breve\bullet}
},
\end{equation}
where $\breve\bullet$ indicates the form   $v^{12}\overline{\frac{1}{\Delta}}$ of weight $12$,
$\breve{\circledast}$ indicates the harmonic form $v^{-12}\overline{R_{-12}^{12}\frac{1}{\Delta}} = \mathfrak{F}_{-12}\frac{1}{\Delta}$,  and $\breve{\odot}$ indicates the
weight $-14$ anti-holomorphic form $v^{14}\bar h$. Again, the omitted arrows are zero.

Here is should be noted that the location of the harmonic Maass form in the diagram changes under the flip, from $\bullet$ in (\ref{5.6})
to $\breve\circledast$ in (\ref{5.7}).  This is due to the harmonicity requirement, i.e.,
that the harmonic Maass form is annihilated by the composition $R_{r-2}\circ L_k$.

\subsubsection{Case I(d)}

Here we want to realize a copy of the $(\g,K)$-module $I(\nu)$ for $\nu=1-k$ and $k<1$ in the space of automorphic forms.

The simplest example is given by adding a weakly holomorphic form and its flip. To be more precise, let $F\in M_k^!\backslash\{0\}$ and set $G:=F+\mathfrak{F}_k(F)$.
Then, by \eqref{xiflip} and the fact that $F\in M_k^!$,
\begin{align*}
	\xi_kG
	=-\frac{(-4\pi)^{1-k}}{(-k)!}D^{1-k}F\neq 0.
\end{align*}
Moreover, using \eqref{Dflip} yields
\begin{align*}
	D^{1-k}G=D^{1-k}F\neq0.
\end{align*}

A second example can be constructed using Eisenstein series and we refer to \cite{kudla.yang.eis} as a convenient reference.
As in Section 3 of loc.cit., for $r\in2\Z$, we let
\begin{equation}\label{def-Eis}
E_{r,s}(\tau) = \sum_{\gamma=\left(\begin{smallmatrix}a&b\\c&d\end{smallmatrix}\right)\in \Gamma_{\infty}\backslash \SL_2(\Z)}
(c\tau+d)^{-r} \, \operatorname{Im}(\gamma\tau)^{\frac12(s+1-r)}.
\end{equation}
This series is absolutely convergent for ${\rm Re}(s)>1$ and defines an automorphic form of weight $r$.
Its Fourier series is given by Proposition~3.1 of loc.cit.: 
\begin{align*}
E_{r,s}(\tau) &= v^\beta + v^{\beta-s}\,2\pi i^r \frac{2^{-s}\Gamma(s)}{\Gamma(\alpha)\Gamma(\beta)}\,\frac{\zeta(s)}{\zeta(s+1)} \\
&\quad+\frac{i^r (2\pi)^{s+1}\,v^\beta}{\Gamma(\a)\,\zeta(s+1)} \sum_{m=1}^\infty \sigma_s(m)\, \Psi(\beta,s+1;4\pi m v)\,q^m \\
&\quad+ \frac{i^r (2\pi)^{s+1}\,v^\beta}{\Gamma(\beta)\,\zeta(s+1)}\sum_{m=1}^\infty \sigma_s(m)\, \Psi(\alpha,s+1;4\pi m v)\,\overline{q^m},
\end{align*}
where $\Gamma(s)$ is the usual Gamma function, $\zeta(s)$ the Riemann zeta-function,
\[
\alpha:= \frac12(s+1+r),\qquad
\beta:=\frac12(s+1-r),\qquad
\sigma_s(m):= \sum_{d\mid m} d^s,
\]
and $\Psi(a,b;z)$ is the confluent hypergeometric function, given by
$$\Psi(a,b;z) := \frac{1}{\Gamma(a)}\int_0^\infty e^{-zt}\,(1+t)^{b-a-1}\,t^{a-1}\,dt.$$
Note that we take $\Psi(0,b;z)=1$ (cf. \cite{kudla.yang.eis}, the equation after (2.20)).

Recall that the principal series representation\footnote{Here we write $s$ for $\nu$
and take $\varepsilon=0$.}
$I(s)$ is spanned by the functions $\phi_r(s)$, defined in (\ref{ind-K-type}). For ${\rm Re}(s)>1$,
we obtain a linear map
\begin{equation*}
\widetilde{E(s)}: I(s) \lra A(G;\Gamma), \qquad \phi_r \mapsto \widetilde{E_{r,s}},
\end{equation*}
and this map is $(\g,K)$-intertwining. Thus, by (\ref{PS.action}),
\begin{align}
L_rE_{r,s} &= \frac12(s+1-r) E_{r-2,s},\label{lower-Er} \quad \text{and}\\
R_rE_{r,s} &= \frac12(s+1+r)\,E_{r+2,s}.\label{raise-Er}
\end{align}
In particular, set $\ell=-k$ with $k<0$
and let $s_0 = 1+\ell=\nu$. Then we obtain a $(\g,K)$-intertwining map
$$\widetilde{E_{s_0}}: I(s_0) \lra A(G;\Gamma), \qquad \phi_r \mapsto \widetilde{E_{r,s_0}}.$$
In fact, this map is injective. To see this, note that the constant term of $E_{r,s_0}$ equals
$$v^{1+\frac12(\ell-r)} + v^{-\frac12(\ell+r)}\,
\frac{2\pi i^r\,2^{-\ell-1}\Gamma(\ell+1)}{\Gamma\left(1+\frac12(\ell+r)\right)\Gamma\left(1+\frac12(\ell-r)\right)}\,\frac{\zeta(\ell+1)}{\zeta(\ell+2)},$$
where, if $|r|\ge \ell+2$, the second term is zero due to the pole in the denominator.
Thus $E_{r,s_0}\ne 0$ and these functions are linearly independent since they have distinct weights.
The function
$$
f(\tau) := E_{k,1-k}(\tau)
$$
is the harmonic Maass form of weight $k$ for $\SL_2(\Z)$. Its Fourier expansion is given by
\begin{align*}
E_{\ell+1,-\ell}(\tau) &= v^{\ell+1} + 2\pi\, i^{\ell} \,
2^{-\ell-1}\,\frac{\zeta(\ell+1)}{\zeta(\ell+2)}\\
\nass
{}&\qquad\qquad + \frac{i^{\ell} (2\pi)^{\ell+2}\,v^{\ell+1}}{\zeta(\ell+2)}\sum_{m=1}^\infty \sigma_{\ell+1}(m)\, \Psi(\ell+1,\ell+2;4\pi m v)\,q^m\\
\nass
{}&\qquad\qquad + \frac{i^{\ell} (2\pi)^{\ell+2}\,v^{\ell+1}}{\Gamma(\ell+1)\,\zeta(\ell+2)}\sum_{m=1}^\infty \sigma_{\ell+1}(m)\,
\Psi(1,\ell+2;4\pi m v)\,\overline{q^m}.
\end{align*}
But we have
\begin{align*}
\Psi(\ell+1,\ell+2;z) &= \frac{1}{\Gamma(\ell+1)}\int_0^\infty e^{-z t}\,t^\ell\,dt = z^{-\ell-1},\qquad\text{and}\\
\Psi(1,\ell+2;z) &= e^{z}\,\int_1^\infty e^{-zt} \,t^\ell\,dt =z^{-\ell-1}\,e^{z}\,\int_z^\infty e^{-t} \,t^\ell\,dt  = z^{-\ell-1}\,e^{z}\,\Gamma(\ell+1,z).
\end{align*}
Thus
\begin{align*}
E_{\ell+1,-\ell}(\tau) &= v^{\ell+1} +2\pi\,
2^{-\ell-1}\, i^{\ell} \,\frac{\zeta(\ell+1)}{\zeta(\ell+2)}
+\frac{2\pi \,2^{-\ell-1}\,i^{\ell}  }{\zeta(\ell+2)}\sum_{m=1}^\infty \sigma_{-\ell-1}(m)\, q^m\\
\nass
{}&\qquad\qquad + \frac{2\pi \,2^{-\ell-1}\,i^{\ell} }{\Gamma(\ell+1)\,\zeta(\ell+2)}\sum_{m=1}^\infty \sigma_{-\ell-1}(m)\,
\,\beta_{-\ell}(4\pi m v)\,q^{-m},
\end{align*}
as in (\ref{weak-Fourier}).
The image of this series under $\xi_k= \xi_{-\ell}$ is $1-k$ times the holomorphic Eisenstein series $E_{2-k}(\tau)$ of weight $2-k$, while its image under $R_k^{1-k}$
is a non-zero multiple of $E_{2-k}(\tau)$. \\
We omit the case $k=0$.\\

\begin{remark}
One could simply define, for $k\in -2\N$,
\begin{align*}
\mathcal{Q}_k(\tau)
:=\sum_{\gamma\in\Gamma_\infty\backslash \SL_2(\Z)}v^{1-k}\Big|_k\gamma.
\end{align*}
Then $\mathcal{Q}_k\in H_k^\text{mg}$ and
\begin{align*}
\xi_k\left(\mathcal{Q}_k\right)
&=(1-k)E_{2-k},\\
D^{1-k}\left(\mathcal{Q}_k\right)
&=-(4\pi)^{1-k}(1-k)!E_{2-k}.
\end{align*}
However, we included the description including an $s$-parameter as this adds an extra perspective described in Section~\ref{section.related}.
\end{remark}

\subsection{Case II: $k=1$}

\subsubsection{Case II(a)}
 Any weakly holomorphic modular form $\uf $ of weight $1$ gives an example.\\
\subsubsection{Case II(b)}
 In this case we want to realize the extension of $(\g,K)$-modules
\begin{equation}\label{tiny.ext}
0\lra \LDS^-(0) \lra M\left(\tf\right) \lra \LDS^+(0)\lra 0,
\end{equation}
which, in classical language, is equivalent to finding
a harmonic Maass form $\uf $ of weight $1$ whose image under $\xi_1$ is
a holomorphic modular form of weight $1$.

An example for this case is given by the derivatives of incoherent Eisenstein series of weight $1$, constructed in \cite{kry.tiny}.
To describe this, we fix an imaginary quadratic field 
$\kay$ of prime discriminant $-D$ with $D\equiv 3\pmod{4}$ and $D>3$. Then define the pair of weight $1$ Eisenstein series by
\begin{align*}
E^{\pm}_s(\tau): &= v^{\frac{s}2}\sum_{\gamma=\left(\begin{smallmatrix}
		a & b\\
		c & d
	\end{smallmatrix}\right)\in \Gamma_\infty\backslash \SL_2(\Z)} \P_D^{\pm}(\gamma)(c\tau+d)^{-1}|c\tau+d|^{-s},\\  
\noalign{\noindent where}
\P_D^{\pm}(\gamma)
:&=\begin{cases}
	\chi_D(a)&\text{if $D\mid c$,}\\
	\pm i D^{-\frac12}\,\chi_D(c)&\text{if $(c,D)=1$,}
\end{cases}
\end{align*}
with $\chi_D $ the quadratic character $D$ associated to $\kay$. These series are absolutely convergent for ${\rm Re}(s)>1$ and
have a meromorphic continuation to the whole $s$-plane. Let $L(s,\chi_D)$ be the usual Dirichlet L-series associated to $\chi_D$.
Then, the analytic continuations in $s$ of the normalized series  
$$
\widehat{E}^{\pm}_s(\tau) := \left(\frac{D}{\pi}\right)^{\frac{s+1}2}\Gamma\left(\frac{s+1}2\right)\,L(s,\chi_D)\,E^{\pm}_s(\tau)
$$
satisfies the functional equation
$$
\widehat{E}^{\pm}_{-s}= \pm \widehat{E}^{\pm}_s.
$$
We refer to $\widehat{E}^{+}_s$ (resp. $\widehat{E}^{-}_s$) as the {\it coherent} (resp. {\it incoherent}) Eisenstein series associated to $\kay$ (see \cite{kry.tiny}).
For the coherent Eisenstein series $\widehat{E}^{+}_s$, we have
$$
\frac12\, \widehat{E}^{+}_0(\tau) = h_{\kay} + 2 \sum_{n=1}^\infty \rho(n)\,q^n = \sum_{\frak a} \vartheta_{\frak{a}}(\tau),
$$
where the sum runs over representatives $\frak a$ for the ideal classes,
$h_{\kay}$ is the class number of $\kay$, and $\rho(n)$ is the number of integral ideals of norm $n$.
Here
$$
\vartheta_{\frak{a}}(\tau):= \sum_{n\in \frak a} q^{\frac{N(n)}{N(\frak a)}}
$$
is Hecke's weight $1$ theta series for the ideal class of the fractional ideal $\frak a$.
Due to the functional equation, the incoherent series $\widehat{E}^-_s$ vanishes at $s=0$, so we instead consider the
function
$$
\phi(\tau) := \frac12\,\frac{\d}{\d s}\left[ \widehat{E}^-_{s}(\tau)\right]_{s=0} =: a_v(0) + \sum_{n=-\infty}^1 a_v(n) \,q^{n} + \sum_{n=1}^\infty a(n)\, q^n.
$$
According\footnote{Note that we have changed the sign compared to \cite{kry.tiny}.} to  Theorem 1 of \cite{kry.tiny}, we have
$$
a_v(n)
=\begin{cases}
	-2\log(D)\,(\ord_D(n)+1)\rho(n) - 2\sum_{p\ne D} \log(p)\,(\ord_p(n)+1)\,\rho\left(\frac{n}{p}\right)&\text{for $n>0$},\\
	h_{\kay}\, \left(\log(D)  + \frac12\frac{\Lambda'(1,\chi_D)}{\Lambda(1,\chi_D)} +\log(v)\right) &\text{for $n=0$},\\
	-2\,\rho(-n)\,\beta_1(4\pi |n|v)&\text{for $n<0$},
\end{cases}
$$
where $\Lambda(s,\chi_D) := \pi^{-\frac12(s+1)}\,\Gamma(\frac12(s+1))\,L(s,\chi_D)$,
and $\beta_1(x)=\Gamma(0,x)$ is the incomplete $\Gamma$-function. 
Here we slightly abuse notation and write $a_v(n) = a(n)$ if $n>0$.

Applying the $\xi$-operator, we obtain
$$
\xi_1\phi= \frac12\, \widehat{E}_0^+.
$$
Thus, taking the function $\uf (\tau) = \phi(\tau)$ on $\H$ with corresponding function $\tf(g) = j(g,i)^{-1}\,\phi(g(i))$ on $G$, we obtain the desired extension (\ref{tiny.ext})
of $(\g,K)$ modules,
where the submodule $\LDS^-(0)$ is generated by the function
$$
j(g,i)\,\frac12\, \overline{\widehat{E}_0^+\left(g(i)\right)}.
$$

\subsection{Case III: $k>1$}
\subsubsection{Case III(a)}   The function $\tf$ on $G$ corresponding to any holomorphic cusp form $\uf $ of weight $k$ generates a copy of $\DS^+(\nu)$ and hence provides an example
for this case. \\
\subsubsection{Case III(b)} For $k=2$, let
\begin{equation}\label{E2star}
\uf (\tau) = E_2^\ast(\tau) := 1 - 24\sum_{n=1}^\infty \sum_{d\vert n}d\,q^n-\frac{3}{\pi v} ,
\end{equation}
be the classical non-holomorphic Eisenstein series of weight $2$. Then
$L_2E^*_2 = \frac{3}{\pi}$ and the
$(\frak g,K)$-module $M(\tf)$ generated by the corresponding $\tf$ gives an extension
$$0\lra \FD(1) \lra M\left(\tf\right) \lra \DS^+(1)\lra 0,$$
with the trivial representation $\FD(1)$ as a submodule and the holomorphic discrete series $\DS^+(1)$ of weight $2$ as quotient.
This example was discussed in connection with the theory of nearly holomorphic modular forms in \cite{PSS}.\\

In general, in case III(b), we want an extension
$$0\lra \FD(\nu) \lra M\left(\tf\right) \lra \DS^+(\nu)\lra 0,$$
and thus, in classical language, a harmonic Maass form $\uf $ of weight $k = \nu+1>2$ such that $\xi_k\uf\ne 0$ but $R_{2-k}^{k-1} \xi_k\uf =0$, so that
$\xi_k\uf$ generates the ``finite dimensional'' piece.
This cannot happen for scalar-valued forms of weight $k>2$. Indeed, the form $h=\xi_k\uf $ is a weakly holomorphic form of weight $2-k <0$.  
As in the proof of Lemma~\ref{lemma6.5}, the vanishing of
$R_{2-k}^\nu h$ implies that $h$ is a constant and hence $h=0$ if $k>2$.
Thus, in the scalar-valued case,  only case III(b) with $\nu=1$, i.e., the $E_2^*$ example, can occur.

However, in the vector-valued case, we can construct more examples as follows.
As in the introduction, define a polynomial-valued function 
\begin{equation}\label{Ek}
	E^*_{m+2}(\tau):= \sum_{r=0}^m \frac{1}{r+1} {m\choose r}\,\ue_{r,m-r}(\tau)\,R^r\,E^*_2(\tau),
\end{equation}
a linear combination of products of the polynomials $e_{r,m-r}(\tau)$ and the $R^r\,E^*_2(\tau)$'s.
Here and in the following, we often omit the subscript on the raising and lowering operators to lighten the notation.
\begin{proposition}\label{prop-of-Ek}
	The polynomial-valued function $E^*_{m+2}$ satisfies
	$$
	E^*_{m+2}(\gamma\tau) = j(\gamma,\tau)^{m+2}\,\rho_m(\gamma)\,E^*_{m+2}(\tau).
	$$
	Moreover
	$$
	L_{m+2}\,E^*_{m+2} = \frac{3}{\pi} \,\ue_{0,m}.
	$$
\end{proposition}
\begin{proof}
	For the transformation law, we have
	\begin{align*}
		E^*_{m+2}(\gamma\tau) &= \sum_{r=0}^m \frac{1}{r+1} {m\choose r}\,\ue_{r,m-r}(\gamma\tau)\,R^r\,E^*_2(\gamma\tau)\\
		\noalign{\smallskip}
		{}&=\sum_{r=0}^m \frac{1}{r+1} {m\choose r}\,j(\gamma,\tau)^{m-2r}\rho_m(\gamma)\ue_{r,m-r}(\tau)\,j(\gamma,\tau)^{2+2r}\,R^r\,E^*_2(\tau)\\
		\noalign{\smallskip}
		{}&=j(\gamma,\tau)^{m+2}\,\rho_m(\gamma)\,E^*_{m+2}(\tau),
\end{align*}
	as claimed.
	
	Next, applying the lowering operator and using the fact that
	$$
	R_{k-2}L_k = L_{k+2}R_k+k,
	$$
	we have
	$$
	L\,R^r \,E^*_2 = -2r R^{r-1}\,E^*_2 + RLR^{r-1}\,E^*_2,
	$$
	and hence
	$$
	LR^r\,E_2^* = -r(r+1) R^{r-1}\,E_2^*.
	$$
	Then, for $r\in\N$,
	\[
	L_{m+2}\big(\,\ue_{r,m-r}\,R^r\,E^*_2\,\big)= (r+1)(m-r)\,\ue_{r+1,m-r-1} \,R^r\,E^*_2
	- r(r+1)\,\ue_{r,m-r}\,R^{r-1}\,E^*_2,
	\]
	while,
	$$
	L_{m+2}\big(\,\ue_{0,m}\,E^*_2\,\big)  =  m \,\ue_{1,m-1}\,E^*_2 + \frac{3}{\pi}\,\ue_{0,m}.
	$$
	For $r\in\N_0$, the coefficient of $\ue_{r+1,m-r-1}(\tau) \,R^r\,E^*_2(\tau)$ in $L_{m+2}(\,\ue_{r,m-r}(\tau)\,R^r\,E^*_2(\tau)\,)$ equals
\begin{align*}
\frac{1}{r+1} {m\choose r}(r+1)(m-r) - \frac{1}{r+2} {m\choose r+1}(r+1)(r+2).
\end{align*}
This vanishes for $0\le r<m$ and the claimed lowering identity follows.
\end{proof}
Here is the picture of the corresponding $(\g,K)$-module:
$$
\xymatrix{
{}&{\scriptstyle\ominus} \ar@/^/[r]^{R_{2-k}}&{\circ}\ar@/^/[l]^{L_{4-k}}
\ar@/^/[r]^{R_{4-k}}&\ \cdots\ \ar@/^/[l]^{L_{6-k}}\ar@/^/[r]^{R_{k-6}}&
\ar@/^/[l]^{L_{k-4}}{\circ}\ar@/^/[r]^{R_{k-4}}&
{\scriptstyle\oplus}\ar@/^/[l]^{L_{k-2}}&{\bullet}\ar@/^/[r]^{R_{k}}\ar@/^/[l]^{L_{k}}&{\circ}\ar@/^/[l]^{L_{k+2}}\ar@/^/[r]^{R_{k+2}}&\cdots\ar@/^/[l]^{L_{k+4}}
}
$$
where $\bullet$ indicates $E^*_{k}$, ${\scriptstyle\oplus}$ indicates $e_{0,m}$ and ${\scriptstyle\ominus}$ indicates $e_{m,0}$.
Note that in all of these pictures there are unspecified, but non-zero, transition constants.

\subsubsection{Case III (c)}
A natural construction of positive weight harmonic Maass forms goes through sesquiharmonic Maass forms. They satisfy condition (1) and (3) 
of harmonic Maass forms but condition (2) is replaced by
\begin{equation}\label{Dk2}
\Delta_{k,2}f=0\qquad\text{with}\qquad \Delta_{k,2}:= -\xi_k\circ\xi_{2-k}\circ\xi_k.
\end{equation}
Let $H_{k,2}^{\rm{mg}}$ be the space of sesqui-harmonic Maass forms. A way to construct sesquiharmonic forms is to differentiate, with 
respect to an additional parameter as we do for weight $1$ in Subsection 5.2.2. In weight $3/2$ this has been done by Duke, 
Imamoglu, and Toth \cite{DIT} in the context of finding a preimage under the $\xi$-operator of the Hirzebruch-Zagier Eisenstein series. Let
\[
H_{k,2}^{\sharp}:=\left\{f\in H_{k,2}:\xi_k(f)\in H_{2-k}^{\sharp}\right\}.
\]
In \cite{BDR} it was shown that the map
\[
\xi_k: H_{k, 2}^{\sharp}\to H_{2-k}^{\sharp}
\]
is surjective. Note that a similar construction can be carried out for all forms mapping to $H_{2-k}$, as we show below.
We use Poincar\'e series and differentiate with respect to an extra $s$-parameter.
To be more precise let for $m\in\Z\setminus\{0\}$
\begin{equation*}
	\mathbb{F}_{k,m}(\tau) := \mathbb{P}_k\left(\psi_{k,m}\right)
\end{equation*}
with
\begin{equation*}
	\psi_{k,m}(\tau):=\left[\frac{\partial}{\partial s}\mathcal{M}_{k,s}(4\pi mv)\right]_{s=\frac{k}{2}}e(mu),
\end{equation*}
where
\[
\mathcal{M}_{k, s}(w):=|w|^{-\frac{k}{2}} M_{\sgn(w)\frac{k}{2}, s-\frac12}(|w|).
\]
Then $\mathbb{F}_{k,m}\in H_{k,2}$ and, with same constant $c_{k, m}\neq 0$,
\begin{equation*}
	\xi_k\mathbb{F}_{k,m}=c_{k, m} F_{2-k,-m}.
\end{equation*}
Now consider in particular
\begin{equation*}
	f := \mathbb{F}_{4,1}.
\end{equation*}
Then
\begin{equation*}
	\xi_4f = c_{4, 1}F_{-2,-1} \in H_{-2}.
\end{equation*}
However, since the space of dual weight, $S_4 =\{0\}$, $\xi_4f \in M^{!}_{-2}$, and thus $f\in H^{\mathrm{mg}}_4$.

\section{Some related examples}\label{section.related}

It is natural
to consider automorphic forms that are not harmonic but are annihilated  by a power of the weight $k$ Laplacian, i.e., with condition (2)
replaced by
\begin{equation}\label{higher-power-harmonic}
\Delta_k^\ell f=0,
\end{equation}
for some $\ell\in\N$.
In particular for harmonic Maass forms we have $\ell=1$ and for sesqui-harmonic forms $\ell=2$. Now, if the minimal power of $\Delta_k$ annihilating $f$ is greater than $1$,  the relations (\ref{multi-one-1}) and (\ref{multi-one-2}) need not hold and, as a result, the $K$-types in $M(\tf)$ can occur with higher multiplicity.
Much more complicated $(\g,K)$-modules can arise.  Here we give a couple of examples, leaving a more systematic analysis to
another occasion.

First, we consider the Eisenstein series of weight $0$ given by (\ref{def-Eis}) with $r=0$.
It has a simple pole at $s=1$ and Laurent expansion
$$2\,\zeta(s+1)\,E_{r,0}(\tau) = \frac{2\pi}{s-1}+ 2\pi\,\big(\,\gamma-\log(2))+\frac{\pi}{2}\,\phi(\tau) + O(s-1),$$
where, by the Kronecker limit formula \cite{lang,siegel}, is as defined in \eqref{Kronecker-fun}.
An easy calculation shows that
$$L_0\phi(\tau) = \frac{\pi}3\,v^2\, \overline{E_2^*(\tau)},\quad\text{and}\quad
R_0\phi(\tau) = \frac{\pi}3\,E_2^*(\tau).$$  
Thus
$$\Delta_0\phi = L_2R_0\phi= 1,$$
so that $\Delta_0^2\phi=0$.  Actually, since $\Delta_0\phi$ is already annihilated by
$R_0$, $\phi$ is annihilated by $\Delta_{0,2}$ as defined in \eqref{Dk2} and thus is sesquiharmonic.
The $(\g,K)$-module generated by $\widetilde{\phi}$ contains the trivial representation of $K$ with multiplicity $2$ and all other even weights with
multiplicity $1$.
This $(\g,K)$-module is pictured in Figure 1 at the end of the introduction,
where $\bullet$ indicates the function $\phi$, $\odot$ indicates the constant function, and the omitted arrows $R_0$ and $L_0$ originating from $\odot$ are zero.
This $(\g,K)$-module has a socle filtration with the trivial representation as maximal semi-simple submodule, the direct sum of the weight $2$ holomorphic and anti-holomorphic discrete series
as intermediate subquotient, and the trivial representation as the unique irreducible quotient.

A little more generally, consider the Laurent expansion of the Eisenstein series (\ref{def-Eis}) at any point $s_0$ with ${\rm Re}(s_0)>1$,
$$E_{\ell,s}(\tau) =\sum_r A_{r,\ell,s_0}(\tau)\,(s-s_0)^r .$$
Note that for $r<0$, $A_{r,\ell,s_0}(\tau)=0$ for all $\ell$, since $E_{\ell,s}(\tau)$ has no pole at $\tau=s_0$ in the half-plane ${\rm Re}(s_0)>1$.
Let $\mathcal E_r(s_0)$ be the span of the functions $A_{r,\ell,s_0}$ for $t\le r$ and $\ell\in 2\Z$ and let $\widetilde{\mathcal E_r(s_0)}$
be the span of their lifts $\widetilde{A_{r,\ell,s_0}}$ to $G$ via (\ref{lifttoG}).

\begin{proposition}\label{surlinmap}
The surjective linear map defined by
\begin{equation*}
\psi_r(s_0): I(s_0) \lra \widetilde{\mathcal E_r(s_0)}\bigg/ \widetilde{\mathcal E_{r-1}(s_0)},\qquad \phi_\ell \mapsto \widetilde{A_{r,\ell,s_0}}
\end{equation*}
is equivariant for the action of $(\g,K)$.
Moreover, this map is an isomorphism.
\end{proposition}
\begin{remark} Proposition \ref{surlinmap} yields that the space $\widetilde{\mathcal E_r(s_0)}$ has a $(\g,K)$-invariant filtration with $I(s_0)$ as subquotients.
The interesting case is if $s_0=k-1$ for an even integer $k>2$, so that the structure of $I(s_0)$ is given by (\ref{Is0-decomp}).
The quotients $\widetilde{\mathcal E_r(s_0)}/\widetilde{\mathcal E_{r-j}(s_0)}$ for $j>1$ reveal a more complicated extension of
$(\g,K)$-modules due to the relations (\ref{Laurent-lower}) and (\ref{Laurent-raise}).  The situation is illustrated in Figure 1 below.
\end{remark}
\begin{proof}[Proof of Proposition \ref{surlinmap}]
Relations (\ref{lower-Er}) and (\ref{raise-Er}) imply that
\begin{align}
L_\ell A_{r,\ell,s_0} &= \frac12(s_0+1-\ell)\,A_{r,\ell-2,s_0} + \frac12\,A_{r-1,\ell-2,s_0}\quad\text{and}\label{Laurent-lower}\\
R_\ell A_{r,\ell,s_0} &= \frac12(s_0+1+\ell)\,A_{r,\ell+2,s_0} + \frac12\,A_{r-1,\ell+2,s_0}.\label{Laurent-raise}
\end{align}
There are analogous relations for the action of $X_\pm$ on the $\widetilde{A_{r,\ell,s_0}}$'s. Moreover, we have
$$\Delta_\ell E_{\ell,s} = \frac14\,(s+1-\ell)(s+1+\ell-2)\,E_{\ell,s},$$
and hence
\begin{equation}\label{delta-rel}
\Delta_\ell A_{r,\ell,s_0}=\frac14(s_0+1-\ell)(s_0+1+\ell-2)\,A_{r,\ell,s_0}
+ \frac12 s_0\,A_{r-1,\ell,s_0} + \frac14\,A_{r-2,\ell,s_0}.
\end{equation}

Now suppose that $s_0 = k-1$ so that  $A_{0,k,s_0}\ne0$ is the standard weight $k$
holomorphic Eisenstein series.  For any $r\in\mathbb{N}_0$, relation (\ref{delta-rel})
then gives
$$
\Delta_k A_{r,k,s_0} = \frac12 s_0\,A_{r-1,k,s_0} + \frac14\,A_{r-2,k,s_0}.
$$
Hence
$$\Delta_k^r \,A_{r,k,s_0} = 2^{-r} s_0^r\,A_{0,k,s_0}\ne 0\quad{\rm and}\quad
\Delta_k^{r+1}A_{r,k,s_0}=0.$$
In particular, the functions $A_{r,k,s_0}$ of weight $k$ are linearly independent as $r$ varies and give examples of functions satisfying
(\ref{higher-power-harmonic}). Moreover, it follows that $A_{r,k}\notin \mathcal E_{r-1}(s_0)$ so that the map $\psi_r(s_0)$ is non-zero.
Also note that, by the raising relation above,
$$R_{k-2}A_{r,k-2,s_0} = (k-1)\,A_{r,k,s_0} + \frac12\,A_{r-1,k,s_0},$$
so that the image of $\phi_{k-2}$ under $\psi_r(s_0)$ is also non-zero. By equivariance, it follows that  $\psi_r(\phi_{2-k})\ne0$.
On the other hand, again by (\ref{delta-rel}),  we have
$$(\Delta_{-k}+k)A_{r,-k,s_0} = \frac12 s_0\,A_{r-1,-k,s_0} + \frac14\,A_{r-2,-k,s_0},$$
so that
\begin{align*}
(\Delta_{-k}+k)^rA_{r,-k,s_0} &= 2^{-r}\,s_0^r\,A_{0,-k,s_0}\quad\text{and} \\
(\Delta_{-k}+k)^{r+1}A_{r,-k,s_0} &= 0.
\end{align*}
But
$$A_{0,-k,s_0}(\tau) = E_{-k,s_0}(\tau) = v^k\sum_{\left(\begin{smallmatrix}
	a & b\\
	c & d
	\end{smallmatrix}\right)\in \Gamma_{\infty}\backslash \SL_2(\Z)}
(c\overline{\tau}+d)^{-k},$$
so we again conclude that the $A_{r,-k,s_0}$'s are linearly independent as $r\in\mathbb{N}_0$ varies. This implies that $\psi_r(\phi_{-k})\ne0$ and
thus that $\psi_r(s_0)$ is an isomorphism.
\end{proof}
\vfill\eject
The following picture summarizes the structure that arises:
\medskip

$$
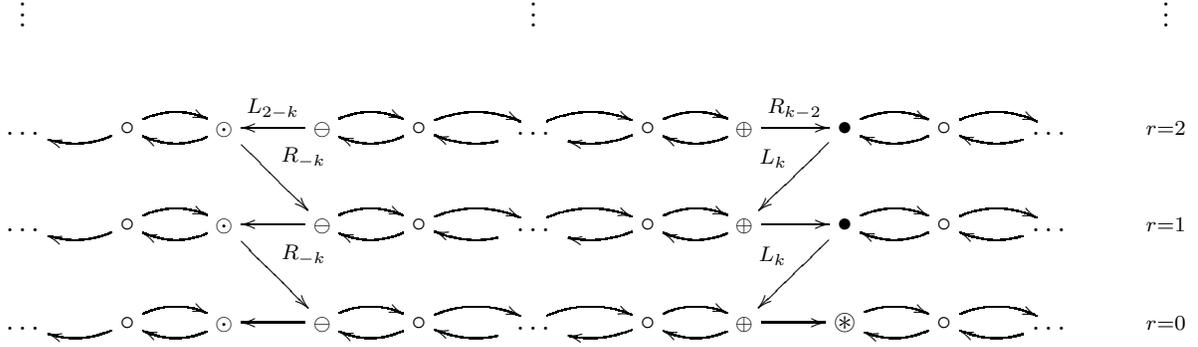
\begin{matrix}\text{\small\bf Figure 2. The $(\g,K)$-module for the Taylor coefficients $A_{r,\ell,s_0}$}\\
\nass
\text{\bf of $E_{\ell,s}(\tau)$ at $s_0=k-1$}
\end{matrix}
$$
$$
\xymatrix{
\vdots&&&&&\vdots&&&&&&\vdots\\
\dots&{\circ}\ar@/^/[l] \ar@/^/[r]&{\scriptstyle\odot}\ar@/^/[l]\ar[rd]^{R_{-k}}&{\scriptstyle\ominus}\ar[l]_{L_{2-k}} \ar@/^/[r]&{\circ}\ar@/^/[l]
\ar@/^/[r]&\ \dots\ \ar@/^/[l]\ar@/^/[r]&
\ar@/^/[l]{\circ}\ar@/^/[r]&
{\scriptstyle\oplus}\ar[r]^{R_{k-2}}\ar@/^/[l]&{\bullet}\ar@/^/[r]\ar[ld]_{L_k}&{\circ}\ar@/^/[l]\ar@/^/[r]&\dots\ar@/^/[l]&\text{$\scriptstyle r=2$}\\
\dots&{\circ}\ar@/^/[l] \ar@/^/[r]&{\scriptstyle\odot}\ar@/^/[l]\ar[rd]^{R_{-k}}&{\scriptstyle\ominus}\ar[l]
\ar@/^/[r]&{\circ}\ar@/^/[l]
\ar@/^/[r]&\ \dots\ \ar@/^/[l]\ar@/^/[r]&
\ar@/^/[l]{\circ}\ar@/^/[r]&
{\scriptstyle\oplus}\ar@/^/[l]\ar[r]&{\bullet}\ar@/^/[r]\ar[ld]_{L_k}&{\circ}\ar@/^/[l]\ar@/^/[r]&\dots\ar@/^/[l]&\text{$\scriptstyle r=1$}\\
\dots&{\circ}\ar@/^/[l] \ar@/^/[r]&{\scriptstyle\odot}\ar@/^/[l]&{\scriptstyle\ominus}\ar[l] \ar@/^/[r]&{\circ}\ar@/^/[l]
\ar@/^/[r]&\ \dots\ \ar@/^/[l]\ar@/^/[r]&
\ar@/^/[l]{\circ}\ar@/^/[r]&
{\scriptstyle\oplus}\ar[r]\ar@/^/[l]&{\circledast}\ar@/^/[r]&{\circ}\ar@/^/[l]\ar@/^/[r]&\dots\ar@/^/[l]&\text{$\scriptstyle r=0$}
}
$$
\medskip

Here the element in the $\ell$th column and $r$th row is $A_{r,\ell,s_0}$. The
element $\circledast$ is the standard weight $k$ holomorphic Eisenstein series $A_{0,k,s_0}$.
The elements denoted with $\bullet$ (resp. $\odot$) have weight $k$ (resp. $-k$)
and are annihilated by powers of $\Delta_k$ (resp. $\Delta_{-k}+k$).
The arrows along rows with $r>0$ indicate maps defined up to a scalar factor and
an element of the same weight  lying in the row below the target, cf. (\ref{Laurent-lower}) and (\ref{Laurent-raise}). The diagonal arrows and the arrows in the
$r=0$ row are maps involving non-zero scaling factors.

Thus, for example, the $(\g,K)$-module generated by $A_{r,k,s_0}$, which is annihilated by $\Delta_k^6$ but no smaller power,
has the holomorphic discrete series $\DS^+(k-1)$ of weight $k$
as unique irreducible quotient. It has as constituents $\DS^+(k-1)$ with multiplicity $6$, $\DS^-(k-1)$ with multiplicity $5$ and $\FD(k-1)$
with multiplicity $5$.

\begin{remark}\label{rem10} It is clear that the structures we discuss here for harmonic and other Maass forms on $\SL_2(\R)$ have natural generalizations
to automorphic forms on other reductive groups, the basic point being that a weakening of the cuspidal or square integrable growth conditions
allow indecomposable but not irreducible Harish-Chandra modules to occur as archimedean components.
The nearly holomorphic modular forms introduced by Shimura, \cite{shimura.nh-1,shimura.nh-2},
and widely studied since are among the important examples.
Recently, in the case of $\text{Sp}_g(\R)$, for genus $g=2$,
Pitale, Saha, and Schmidt \cite{PSS.siegel} used the Harish-Chandra modules associated to such forms to prove a structure theorem for them.
Also for $\text{Sp}_2(\R)$,  Westerholt-Raum \cite{raum.siegel}
explained the use of Harish-Chandra modules to
provide examples of harmonic weak Siegel Maass forms and their holomorphic parts, Siegel mock modular forms.
We make no attempt to give a systematic review of these and related developments and apologize in advance for the many references omitted as a result.
\end{remark}

\end{document}